\documentclass[sn-mathphys,Numbered]{sn-jnl}

\usepackage{graphicx}%
\usepackage{multirow}%
\usepackage{amsmath,amssymb,amsfonts}%
\usepackage{amsthm}%
\usepackage{mathrsfs}%
\usepackage[title]{appendix}%
\usepackage{xcolor}%
\usepackage{textcomp}%
\usepackage{manyfoot}%
\usepackage{booktabs}%
\usepackage{algorithm}%
\usepackage{algorithmicx}%
\usepackage{algpseudocode}%
\usepackage{listings}%

\theoremstyle{thmstyleone}%
\newtheorem{theorem}{Theorem}
\newtheorem{corollary}[theorem]{Corollary}%
\theoremstyle{thmstyletwo}%
\newtheorem{example}{Example}%
\newtheorem{remark}{Remark}%
\theoremstyle{thmstylethree}%
\newtheorem{definition}{Definition}%

\begin{document}

\title[Article Title]{Mathematical Programs Using Tangential Subdifferentials}


\author[1]{\fnm{Shashi Kant} \sur{Mishra}}\email{shashikant.mishra@bhu.ac.in}

\author*[1]{\fnm{Dheerendra} \sur{Singh}}\email{singhdheeren077@gmail.com}
\equalcont{These authors contributed equally to this work.}

\equalcont{These authors contributed equally to this work.}

\affil*[1]{\orgdiv{Department of Mathematics},  \orgaddress{\street{Institute of Science}, \orgname{Banaras Hindu University},  \city{Varanasi}, \postcode{221005}, \state{Uttar Pradesh}, \country{India}}}




\abstract{In this paper, we deal with constraint qualifications, the stationary concept and the optimality conditions for nonsmooth mathematical programs with equilibrium constraints. The main tool of our study is the notion of tangential subdifferentials. Using the notion of tangential subdifferentials, we present constraint qualifications (namely, generalized standard Abadie, MPEC Abadie, MPEC Zangwill, constraint qualifications) and stationary concepts, and also establish relationships between constraint qualifications. Further, we establish sufficient optimality conditions for mathematical programs using tangential subdifferentials and suitable generalized convexity notion. We also give some examples that verify our results.
}




\maketitle

\section{Introduction}\label{sec1}
Mathematical programming with equilibrium constraints (MPEC) is one of the core topics in optimization which was coined by Harker and Pangn \cite{harker-pang1988} in 1988, after Luo et al. \cite{luo-pang-ralph-wu1996} presented a comprehensive study in 1996. There are several different approaches to reformulate MPEC. The stationarity notion (first-order optimality condition) is a very important concept for MPEC, which was first presented by Outrata \cite{Outrata1999},   Scheel et al. \cite{Scheel-Scholtes2000}. There are several types of stationary conditions. Among all the stationary notions, Strong stationary (S-stationary)~\cite{Pang-Fukushima1999,Scheel-Scholtes2000} is one of the strongest stationary point notions for MPEC, because it is similar to standard KKT conditions in standard nonlinear MPEC programs \cite{flegel-kanzow2005}. Then the second most strong stationary condition is Mordukhovich stationary (M-stationary) \cite{guu-mishra-pandey2016,Outrata1999,pandey-mishra2016(c)}. There are also some weaker stationary conditions like Alternative stationary (A-Stationary), Clarke stationary (C-Stationary) \cite{Scheel-Scholtes2000,flegel-kanzow2003}. The `A' might stand for `Alternative' because that describes the properties of Lagrange multiplier or Abadie because it is first occurred when MPEC-ACQ was extended to MPEC see \cite{flegel-kanzow2005}.

There are some special constraint properties called constraint qualifications (Scheel and Scholtes in \cite{Scheel-Scholtes2000}), which are used to find necessary optimality conditions of nonlinear programming problems. In \cite{flegel-kanzow2003} we can see that S-stationary is a necessary optimality condition under  MPEC form of the strict Mangasarian- Fromovitz constraint qualifications \cite{kyparisis1985}. Also in \cite{flegel-kanzow2005,flegel-kanzow-outrata2006} Flegel, Kanzow and Outrata considered various constraint qualifications which confirms being S-stationry point for MPEC in smooth data.

The class of tangentially convex functions was first introduced by Pshenichnyi \cite{phenichnyi1971} almost half a century ago, then comprehensively studied by Lemarechal \cite{lemarechal1986}, he also coined the term `tangentially convex'. The class of tangentially convex functions is very large. It includes all convex functions with open domain, every \'{G}ateaux differentiable functions \cite{laha-singh-pandey-mishra2022} and Clarke regular functions \cite{clarke1983}, MP regular functions \cite{Michel-Penot1992}. In particular, the sum of the convex function and differentiable function gives an example of tangentially convex functions, which are neither convex nor differentiable, in general. Tangential subdifferentials \cite{martinezlegaz2015} is subdifferential for the class of tangentially convex functions. Tangential subdifferentials enjoy rich calculus rules. For instance, given two functions $f$ and $g$ which are tangentially convex at a common point x, one has $\partial^{T}( f + g) (x) = \partial^{T} f (x) + \partial^{T} g (x)$ this additivity property easily follows from the relationship between tangential subdifferentials and directional derivatives, taking into account that the equality $( f + g)^{'}(x, d) = f^{'}(x, d) + g^{'}(x, d)$ holds for every $d\in \mathbb{R}^{n}.$ Tangential subdifferentials is used in optimization mainly in connection with optimality conditions see in~\cite{allevi-ligaz2020,hiriart-ledyaev1996,martinezlegaz2015,mashkoorzadeh-movahedian2019,sisarat-wangkeeree2020,sisarat-wangkeeree2018}.

Several research related to optimality conditions of MPEC have already been done in both cases smooth and nonsmooth see in \cite{flege-kanzow2005,flegel-kanzow-outrata2006,jeyakumar-luc1998,kyparisis1985,li-zhang2006,joshi-mishra2020,singh-pandey-mishra2017,mishra-jaiswal2015}. Our work is inspired by recent work done in \cite{Ansari-movahedian-nobakhtian2016,dutta-chnadra2004}, where authors used convexificators to find optimality conditions and constraint qualifications for nonsmooth MPEC. Convexificators are always closed sets, but not necessarily convex or compact sets, unlike tangential subdifferentials which are  always compact or convex sets. Here we establish optimality conditions for nonsmooth mathematical programming problems with equilibrium constraints for the class of tangentially convex functions using tangential subdifferentials and suitable generalized convexity assumptions. We also present several Abadie-types constraint qualifications and stationary points in terms of tangential subdifferentials and obtain the relationship between them.

This paper is put in order as follows: in section 2, we recall some basic and fundamental definitions also we give some notations and preliminary results that will be used in this paper. Section 3, this section is devoted to constraint qualifications for MPEC and the relations between them. In section 4, we obtain some optimality conditions under generalized convexity assumptions based on tangential subdifferentials. We also present several MPEC stationary point conditions using tangential subdifferentials. Also, we give some examples that verify our results.

\section{Preliminaries}\label{sec2}
In this section, we recall some basic and important definitions that will be used in this sequel.
Throughout in this sequel, $R^{n}$ is usual n-dimensional Euclidean space. We consider $K$ is the nonempty subset of $\mathbb{R}^{n}.$ The convex hull of $K$ is $coK$, the closure of $K$ is $clK$ and the convex cone (includes origin of $R^{n}$) generated by $K$ is  $conK$. The negative and strictly negative polar cones are denoted by $K^{-}$ and $K^{--}$ respectively, for more details \cite{Borwein-Lewis2010,Mordukhovich2006}.

$$ K^{-}=\{u \in R^{n}: \langle u,k \rangle  \leq 0, \ \  \forall k \in K \},$$

$$K^{--}=\{u \in R^{n}: \langle u,k \rangle < 0, \ \  \forall k \in K \}.$$
The contingent cone $T(K,k)$ and regular normal cone $N^{\perp}(K,k)$ at $k \in cl K$ are defined, respectively, as

$ T(K,k)=\{u \in \mathbb{R}^{n}: \exists t_{z} \rightarrow 0 \ \ and \ \ u_{z} \rightarrow u \ \ such \ \ that \ \ k + t_{z}u_{z} \in K  \ \  \forall z \},$\\

$ N^{\perp}(K,k)= T(K,k)^{-} =\{ v \in \mathbb{R}^{n}: \langle u,v \rangle \leq 0 \ \  \forall u \in T(K,k) \}.$      

\begin{definition}
	A function $ \mathcal{J}:\mathbb{R}^{n}\rightarrow \mathbb{R} \cup \{+\infty \} $ is said to be a directionally differentiable at $k \in \mathbb{R}^{n}$ in direction of $d \in \mathbb{R}^{n}$ if and only if the limit
	$$ \mathcal{J}{'}(k,d)=\lim_{h\downarrow 0} \frac{ \mathcal{J}(k + hd ) -\mathcal{J}(k)}{h}$$
	exist and it is finite. Function  $\mathcal{J}$ is said to be a directionally differentiable at $k$ if and only if its directional derivative $  \mathcal{J}{'}(k, d) $ exists and finite for all direction $ d \in \mathbb{R}^{n}.$	
\end{definition}
\begin{definition} \cite{martinezlegaz2015}
	A function  $\mathcal{J}$ is callled tangentially convex at $k \in \mathbb{R}^{n}$ if for every $d \in \mathbb{R}^{n},  \mathcal{J}^{'}(k,d)$ exists, finite and the function $ \mathcal{J}^{'}(k,.):\mathbb{R}^{n} \rightarrow \mathbb{R}$ is convex function in variable $d.$
	Note that, we know that $ \mathcal{J}^{'}(k,d)$ is positively homogeneous so if $ \mathcal{J}^{'}(k,d)$ is tangentially convex at k, then $\mathcal{J}^{'}(k,d)$ is sublinear.
\end{definition}
\begin{definition} \cite{martinezlegaz2015}
	Suppose the function $ \mathcal{J} :\mathbb{R}^{n} \rightarrow \mathbb{R}$ is a tangentially convex function at $k \in \mathbb{R}^{n}.$ Then nonempty convex compact set $\partial^{T} \mathcal{J}(k),$ which is subset of  $\mathbb{R}^{n}$ is called tangential subdifferential  of  $\mathcal{J}$ at k, such that
	$$ \partial^{T} \mathcal{J}(k)=\{ k^{*} \in \mathbb{R}^{n}: \langle k^{*},d \rangle \leq  \mathcal{J}^{'}(k,d), \ \ \forall d \in \mathbb{R}^{n}\} $$
	which is equivalent to $ \mathcal{J}^{'}(k,d)=max_{k^{*} \in \partial^{T} \mathcal{J}(k)}\langle k^{*},d \rangle.$
\end{definition}
\begin{remark}
	If function is tangentially convex at any point, then there exist $\partial^{T}(.)$ tangential subdifferentials, which is nonempty, compact and convex set \cite{martinezlegaz2015}.
\end{remark}
\begin{example}
	Suppose $ \mathcal{J}:\mathbb{R}^{2} \rightarrow \mathbb{R}$ is real valued function s.t.
	\[
	\mathcal{J}(k_{1},k_{2})=\left\{ \begin{array}{ll}
	\frac{k_{1}^{3}}{k_{2}}+k_{1}, & \ \ (k_{1},k_{2}) \neq (0,0) \\
	0 & k_{1}=0 \ \ or \ \  k_{2}=0 \\
	\end{array}
	\right.
	\]
	Clearly that function  $\mathcal{J}$ is not continuous at 0=(0,0). The directional derivative of  $\mathcal{J}$ at k=(0,0) in direction $d=(d_{1},d_{2})$ is $ \mathcal{J}^{'}(0,d)=d_{1}$ which is convex function. So  $\mathcal{J}$ is tangentially convex function at k=(0,0). Then tangential subdifferential of function  $\mathcal{J}$ at k=(0,0), is 
	$$\partial^{T}\mathcal{J}(0)=\{ (1,0)\}.$$
\end{example}
\begin{example}
	$ \mathcal{J}: \mathbb{R}^{2} \rightarrow \mathbb{R}$ is defined as $ \mathcal{J}(k_{1},k_{2})=|k_{1}|+k_{2}^{2},$ is a tangentially convex function at k=(0,0) then $ \mathcal{J}^{+}(0,d)=|d_{1}|$  and $ \mathcal{J}^{-}(0,d)=|d_{1}|$ are upper and lower Dini derivatives of function respectively in the direction of $d=(d_{1} d_{2})$, also the directional derivative of the function is $ \mathcal{J}(0,d)=|d_{1}|$.
	
	Tangential subdifferentials of the function $\mathcal{J}$ at k=(0,0) is $$\partial^{T} \mathcal{J}(0)=\{(1,0),(-1,0),(0,0)\}.$$
\end{example}
Now we will define generalized convexity notion for tangentially convex functions, as $\partial^{T}$-pseudoconvexity and  $\partial^{T}$-quasiconvexity.
\begin{definition} 
	Let $ \mathcal{J}:\mathbb{R}^{n} \rightarrow \mathbb{R} \cup \{\infty\}$ be a tangentially convex function at $k \in \mathbb{R}^{n}$, then  $\mathcal{J}$ is said to be a $\partial^{T}$-pseudoconvex function at k if, for all $t \in \mathbb{R}^{n}$ such that 
	$$  \mathcal{J}(t) <  \mathcal{J}(k) \ \ implies \ \  that \ \ \langle \xi, t-k \rangle < 0, \ \  \forall \xi \in \partial^{T} \mathcal{J}(k).$$
	Function  $\mathcal{J}$ is said to be $\partial^{T}$-pseudoconcave function at k if, - $\mathcal{J}$ is $\partial^{T}$-pseudoconvex at k.
\end{definition}
\begin{definition}
	Let $ \mathcal{J}:\mathbb{R}^{n} \rightarrow \mathbb{R} \cup \{\infty\}$ be a tangentially convex function at $k \in \mathbb{R}^{n}$, then  $\mathcal{J}$ is said to be a $\partial^{T}$-quasiconvex function at k if, for all $t \in \mathbb{R}^{n}$ such that
	$$  \mathcal{J}(t) \leq  \mathcal{J}(k) \ \ implies \ \ that \ \ \langle \xi, t-k \rangle \leq 0, \ \ \forall \xi \in \partial^{T} \mathcal{J}(k).$$
	$\mathcal{J}$ is said to be $\partial^{T}$-quasiconcave function at k if,  $-\mathcal{J}$ is $\partial^{T}$-quasiconvex at k.	
\end{definition}
\begin{example}
	$\mathcal{J}:\mathbb{R}^{2} \rightarrow \mathbb{R}$ such that $\mathcal{J}(k)= -e^{k_{1}+k_{2}}$, which is not convex but tangentially convex and $\partial^{T}$-pseudoconvex function at 0=(0,0).
\end{example}
Mathematical programming problem with equilibrium constraints MPEC,
\begin{align*}
(MPEC)(P) \ \ &min \ \  \mathcal{J}(k) \\
&\ \ s.t. \ \   \ell(k) \leq 0,  \  \  \hbar(k)=0, \\
&\ \ \ \ G(k)\geq 0,   \ \  H(k) \geq 0, \ \ G(k)^{T}H(k)=0,
\end{align*}
where  $ \mathcal{J}:\mathbb{R}^{n} \rightarrow \mathbb{R},  \ \   \ell:\mathbb{R}^{n} \rightarrow \mathbb{R}^{p},  \ \  \hbar:\mathbb{R}^{n} \rightarrow \mathbb{R}^{q},  \ \  G:\mathbb{R}^{n} \rightarrow \mathbb{R}^{m},  \ \  H:\mathbb{R}^{n} \rightarrow \mathbb{R}^{m}$ are tangentially convex functions.\\
Then the feasible solution set of MPEC is defined as,
$$K:=\{k \in \mathbb{R}^{n} | \ell(k)\leq 0, \ \ \hbar(k)=0, \ \ G(k)\geq 0, \ \ H(k)\geq 0,\ \ G(k)^{T}H(k)=0 \}.$$
For any feasible point $k^{*} \in K$ we define the index set as follow:\\
$P_{ \ell}:=\{ i:  \ell_{i}(k^{*})=0 \},$\\
$\Theta :=\Theta(k^{*}):=\{ i: G_{i}(k^{*})=0, \ \ H_{i}(k^{*}) > 0\},$ \\
$\Omega :=\Omega(k^{*}):=\{ i: G_{i}(k^{*})=0,\ \ H_{i}(k^{*}) = 0 \},$ \\
$\Upsilon := \Upsilon(k^{*}) :=\{ i: G_{i}(k^{*})>0,\ \ H_{i}(k^{*}) = 0 \},$\\
where $\Omega$ is called degenerate set. If $\Omega$ is empty that means vector $k^{*}$ satisfies the strict complementarity condition. In this sequel we will discuss on nonempty $\Omega$.

\section{Constraint Qualifcations}\label{sec3}
This section is devoted to Abadie-types constraint qualifications and some other constraint qualifications \cite{Ansari-movahedian-nobakhtian2016}. We present constraint qualifications in terms of the tangential subdifferentials with nonsmooth data. MPEC Abadie constraint qualifications (ACQ) was first introduced by Flegal et al. for smooth MPEC in \cite{flegel-kanzow2002}, then Movahedian et al. introduced MPEC Abadie constraint qualifications for nonsmooth MPEC case in \cite{Movahedian-Nobakhtian2009}. Finally, we will establish some relationships between these constraint qualifications.

Now, we will recall some notations for a class of tangentially convex functions, in terms of tangential subdifferentials.

$ \ell= \bigcup_{i \in I_{ \ell}} \ \ \partial^{T}  \ell_{i}(k^{*}),$\\

$ \hbar= \bigcup_{i=1}^{q} \ \ \partial^{T}   \hbar_{i}(k^{*})\cup \partial^{T}(- \hbar_{i})(k^{*}), $\\

$G_{\Theta}= \bigcup_{i \in \Theta} \ \ \partial^{T}  G_{i}(k^{*})\cup \partial^{T}(-G_{i})(k^{*}), $\\

$G_{\Omega}= \bigcup_{i \in \Omega} \ \ \partial^{T}  G_{i}(k^{*}),$\\

$H_{\Upsilon}= \bigcup_{i \in \Upsilon}\ \ \partial^{T}  H_{i}(k^{*})\cup \partial^{T}(-H_{i})(k^{*}), $\\

$H_{\Omega}= \bigcup_{i \in \Omega}\ \ \partial^{T}  H_{i}(k^{*}),$\\

$(GH)_{\Omega}= \bigcup_{i \in \Omega} \ \ \partial^{T} (-G_{i})(k^{*})\cup \partial^{T}(-H_{i})(k^{*}), $\\

$$ \Pi(k^{*}):=  \ell^{-} \cap  \hbar^{-} \cap G^{-}_{\Theta} \cap H^{-}_{\Upsilon} \cap(GH)^{-}_{\Omega},$$
$$ \Psi(k^{*})= \ell^{-} \cap  \hbar^{-} \cap G^{-}_{\Theta} \cap H^{-}_{\Upsilon} \cap(GH)^{-}_{\Omega} \cap (G^{-}_{\Omega} \cup H^{-}_{\Omega}). $$

Now by using the above notations, we present following constraint qualifications \cite{Ansari-movahedian-nobakhtian2016} in terms of tangential subdifferentials.
\begin{definition}\textbf{(GS Abadie Constraint Qualification)}
	Suppose $k^{*}$ is a feasible point of MPEC problems, and all constraint functions are tangentially convex at $k^{*}$. Then generalized standard Abadie constraint qualification holds at $k^{*}$, if at least one of the dual sets used in the definition of $\Pi(k^{*})$  is nonempty and holds
	$$\Pi(k^{*}) \subset T(K,k^{*}).$$
\end{definition}
\begin{definition}\textbf{(MPEC Abadie Constraint Qualification)}
	Suppose $k^{*}$ is a feasible point of MPEC problems, and all constraint functions are tangentially convex at $k^{*}$. Then MPEC Abadie constraint qualification holds at $k^{*}$, if at least one of the dual sets used in the definition of $\Psi(k^{*})$  is nonempty and holds
	$$\Psi(k^{*}) \subset T(K,k^{*}).$$	
\end{definition}
Since $\Pi(k^{*}) \subset \Psi(k^{*})$, GS Abadie constraint qualification is also MPEC Abadie constraint qualification.\

Now, we present some other constraint qualifications, the MPEC Zangwill \cite{Zangwill1969} and MPEC weak reverse convex \cite{ye-2005} constraint qualifications, using tangential subdifferentials. For this we recall some notation from nonlinear optimization.\\
The cone of feasible direction of K at $k^{*}$ is defined as:
\begin{equation*}
Dcon(K,k^{*})=\{d \in \mathbb{R}^{n} : \exists \delta > 0 \ \ s.t.  \ \ k^{*}+\lambda d \in K, \ \  \forall \lambda \in (0, \delta) \}.
\end{equation*}

The cone of attainable direction of $K$ at $k^{*}$ is given by

\begin{align*}
Acon(K,k^{*})=\Bigg\{\ d \in \mathbb{R}^{n} : \exists \delta > 0 \  and  \alpha :\mathbb{R} \rightarrow \mathbb{R}^{n} \ \ s.t. \ \ \alpha(\lambda) \in K, \\  \forall \lambda \in (0, \delta), \ \alpha(0)=0, \ \ lim_{\lambda \downarrow 0} \frac{\alpha(\lambda)-\alpha(0)}{\lambda}=d \Bigg\}.
\end{align*}

\begin{definition}\textbf{MPEC Zangwill Constraint Qualifications}
	Suppose $k^{*}$ is a feasible point of MPEC. Let all constraint functions be tangentially convex at $k^{*}$. Then we say that the MPEC Zangwill constraint qualification is satisfied at $k^{*}$  if at least one of the dual sets used in definition $\Psi(k^{*})$ is nonzero and $$\Psi(k^{*}) \subseteq cl Dcon(K,k^{*}).$$	
\end{definition}
	Clearly $Dcon(K,k^{*}) \subseteq  Acon(K,k^{*}) \subseteq T(K,k^{*})$ and $T(K,k^{*})$ is closed, that means,
\begin{equation*}
MPEC~~Zangwill \ \ CQ \implies MPEC~~Abadie\ \ CQ.
\end{equation*}	
There are some other constraint qualifications in MPEC, like weak reverse convex constraint qualification which is expressed in terms of generalized convexity notion.
Here to define these constraint qualifications, we use new generalized convexity assumption based on tangential subdifferentials.
\begin{definition}\textbf{(MPEC Weak Reverse Convex Constraint Qualification)}
	Suppose $k^{*}$ is a feasible point of MPEC(P). Then $k^{*}$ is said to be satisfy reverse convex constraint qualification if, $\ell_{i} (i \in I_{\ell})$ are $\partial^{T}$-pseudoconcave at $k^{*}$, $ \ell_{i} (i \notin I_{ \ell})$, $G_{i} (i \in \Upsilon)$ and $H_{i} (i \in \Theta)$ are continuous at $k^{*}$ and $h_{i}(i=1...,q), G_{i} (i \in \Theta \cup \Omega)$ and $H_{i} (i \in \Upsilon \cup \Omega)$ are $\partial^{T}$- pseudoaffine at $k^{*}.$
\end{definition}
\begin{theorem}
	Every MPEC weak reverse convex constraint qualification is also MPEC Zangwill constraint qualification.
\end{theorem}
\begin{proof}
	Suppose $k^{*}$ is a feasible point of MPEC, and satisfies the MPEC reverse convex constraint qualification at $k^{*}$.If we take $d \in \Psi(k^{*})$, then we get for any $i \in I_{\ell}$:
	$$ \langle \xi,d \rangle \leq 0, \ \  \forall \xi \in \partial^{T}\ell_{i}(k^{*}).$$
	Thus, by using dini derivative's definition, we have
	$$0 \leq \lim\limits_{h\to 0}inf  \frac{- \ell_{i}(k^{*}+hd)}{h}\leq \lim\limits_{h\to 0}sup \frac{- \ell_{i}(k^{*}+hd)}{h}.$$
	Hence,
	$$0 \leq(- \ell_{i})^{+}(k;d) \leq  \sup\limits_{\xi \in \partial^{T}(- \ell_{i})(k^{*})} \langle \xi,d \rangle.$$
	Then, there exist $\xi_{0} \in \partial^{T}(- \ell_{i})(k^{*})$ such that, $\langle \xi_{0},d \rangle \geq 0.$ Now using $\partial^{T}$-pseudoconcavity of each $ \ell_{i}(i \in I_{ \ell})$, we have for all $h > 0,$
	$$ \ell_{i}(k^{*}+hd) \leq 0 =  \ell_{i}(k^{*}), \ \ for \ \ each \ \  i \in I_{ \ell}.$$
	Continuity of all $ \ell_{i}(i \notin I_{ \ell})$ at $k^{*}$  we have,
	$$ \ell_{i}(k^{*}+hd) < 0, \ \ for \ \ each \ \ i \notin I_{ \ell} \ \and 
	\ \ h>0 \ \ small \ \  enough \ \  .$$
	By appling same argument, we can find the scalar $\delta$ such that for $h \in (0, \delta)$,\\
	$ \ell_{i}(k^{*}+hd) \leq 0, i=1...,p$,\\
	$ \hbar_{i}(k^{*}+hd) = 0, i=1...,q$,\\
	$G_{i}(k^{*}+hd) = 0, \ \ H_{i}(k^{*}+hd) > 0, \ \ \forall i \in  \Theta,$\\
	$H_{i}(k^{*}+hd) = 0, \ \ G_{i}(k^{*}+hd) > 0, \ \ \forall i \in  \Upsilon,$\\
	$G_{i}(k^{*}+hd) \geq 0, \ \  H_{i}(k^{*}+hd) \geq 0,\ \  G_{i}(k^{*}+hd)H_{i}(k^{*}+hd)=0,  \ \ \forall i \in  \Omega,$
	which implies that $d \in Dcon(K,k^{*}).$ Proof complete.
\end{proof}

\begin{remark}
	Our definitions of constraints qualifications do not directly generalize the smooth version. For smooth function any set that contains the gradient is tangential subdifferential $\partial^T \mathcal{J}(k)= \{\nabla  \mathcal{J}(k)\}$ and it is unique. Therefore, in differential case, tangential subdifferentials of $\Pi(k^{*})$ and $\psi(k^{*})$ can be replaced by set containing $\{\nabla  \mathcal{J}(k)\}$, then all above constraint qualifications are reduced to smooth notion as in \cite{ye-2005,flegel-kanzow2003}.
\end{remark}

\section{Optimality Condition}\label{sec3}
In this section, we obtain some necessary and sufficient optimality conditions for MPEC with the above constraint qualifications and extended stationary point conditions. First,  we present several extended versions of stationary point conditions \cite{Ansari-movahedian-nobakhtian2016} in terms of tangential subdifferentials and suitable generalized convexity assumptions.\
Now we are presenting our dual stationary notions for nonsmooth MPEC, in terms of tangential subdifferentials.\

\begin{definition}\textbf{GA(Generalised Alternatively)-Stationary Point}
	A feasible point $k^{*}$ of MPEC, where all functions are tangentially convex, is said to be GA(Generalised Alternatively )-stationary point if, there are vectors 
	$\lambda=(\lambda^{ \ell},\lambda^{ \hbar},\lambda^{G},\lambda^{H}) \in \mathbb{R}^{p+q+2m}$ and $\mu=(\mu^{ \hbar},\mu^{G},\mu^{H}) \in \mathbb{R}^{q+2m}$ such that following holds:\\
	\begin{align}\label{1}
	0 \in \partial^{T} \mathcal{J}(k^{*})+&\sum_{i\in I_{ \ell}} \lambda_{i}^{ \ell}\partial^{T}( \ell_{i}(k^{*}))+\sum_{j= 1}^{q}[\lambda_{j}^{ \hbar}\partial^{T}( \hbar_{j}(k^{*}))+\mu_{j}^{ \hbar}\partial^{T}(- \hbar_{j})(k^{*})] \nonumber \\ &+\sum_{i=1}^{m}[\lambda_{i}^{G}\partial^{T}(-G_{i})(k^{*})+\lambda_{i}^{H}\partial^{T}(-H_{i})(k^{*})]\nonumber\\&+\sum_{i=1}^{m}[\mu_{i}^{G}\partial^{T}(G_{i})(k^{*})+\mu_{i}^{H}\partial^{T}(H_{i})(k^{*})],
	\end{align}
	\begin{equation}\label{2}
	\lambda^{ \ell}_{I_{ \ell}} \geq 0, \ \  \lambda^{ \hbar}_{j}, \mu^{ \hbar}_{j} \geq 0, \ \  j=1,...,q, \ \ \lambda^{G}_{i},\lambda^{H}_{i}, \mu^{G}_{i},\mu^{H}_{i}\geq 0, \ \ i=1,...,m,
	\end{equation}
	\begin{equation}\label{3}
	\mu^{G}_{\Upsilon}=\mu^{H}_{\Theta}=\lambda^{G}_{\Upsilon}=\lambda^{H}_{\Theta}=0,
	\end{equation}
	with $$\mu^{G}_{i}=0  \ \ or \ \ \mu^{H}_{i}=0, \ \ \forall i \in \Omega.$$
\end{definition}
\begin{definition}\textbf{GS(Generalised Strong)-stationary point}
	A feasible point $k^{*}$ of MPEC, where all functions are tangentially convex, is called GS(Generalised Strong)-stationary point if, there are vectors $\lambda=(\lambda^{ \ell},\lambda^{ \hbar},\lambda^{G},\lambda^{H}) \in \mathbb{R}^{p+q+2m}$ and $\mu=(\mu^{ \hbar},\mu^{G},\mu^{H}) \in \mathbb{R}^{q+2m}$ such that it satisfies the above Conditions (\ref{1}-\ref{3}) together with following condition:
	$$\mu^{G}_{i}=0, \ \ \mu^{H}_{i}=0, \ \ \forall i \in \beta.$$
\end{definition}
\begin{remark}
	If all the functions are differentiable then tangential subdifferentials can be repleced by usual derivative in above stationary concept.
\end{remark}
So by above definitions it is clear that:
$$GS-stationary~~point  \implies GA-stationary~~point.$$
Now, we are focusing on necessary and sufficient optimality conditions for nonsmooth MPEC, under the notion of tangential subdifferentials.
\begin{theorem}\label{5.1}
	Let $k^{*}$	be a local optimal solution of MPEC. Suppose  $\mathcal{J}$ is tangentially convex and locally Lipschitz at $k^{*}$. Assume that GS-ACQ at $k^{*}$ and the cone 
	\begin{equation*}
	\Delta =cone \ \  \ \  \ell + cone \ \  \ \  \hbar + cone\ \   \ \  G_{\Theta} + cone  \ \   \ \ H_{\Upsilon} + cone \ \ \ \  (GH)_{\Omega},
	\end{equation*}
	is closed. Then $k^{*}$ is GS-stationary point.
\end{theorem}
\begin{proof}
	Firstly, we will proof that
	$$ 0 \in  \ \ \partial^{T} \mathcal{J}(k^{*})+\Delta.$$
	Suppose it is not possible, then $ \ \ \partial^{T} \mathcal{J}(k^{*}) \cap -\Delta = \phi.$ Since $ \partial^{T} \mathcal{J}(.)$ is compact convex set. Also $\Delta$ is closed and bounded set, then by using convex seperation theorem, there exists a non-zero vector $\nu  \in \mathbb{R}^{n}$ with real number $\rho$ such that,
	$$\langle \xi,\nu \rangle < \rho < \langle \eta,\nu \rangle, \ \ \forall \xi \in \partial^{T} \mathcal{J}(k^{*}), \ \ \eta \in -\Delta.$$ 
	Since -$\Delta$ is cone, that means $\rho =0$, and
	\begin{equation}\label{4}
	\langle \xi,\nu \rangle < 0, \ \  \forall \xi \in  \partial^{T} \mathcal{J}(k^{*}).
	\end{equation}
	Since  $\partial^{T} \mathcal{J}(k^{*})$ is bounded at $k^{*}$, then by (\ref{4}), we have $ \mathcal{J}^{+}(k^{*},\nu) < 0.$ Hence, there exists $\delta > 0 $ such that 
	\begin{equation}\label{5}
	\mathcal{J}(k^{*}+h\nu)<  \mathcal{J}(k^{*}), \ \ \forall h \in (0, \delta).
	\end{equation}
	Also applying the properties of cone on $\Delta,$ we have
	$$\langle \eta,\nu \rangle\leq 0, \ \ \forall \eta \in \Delta.$$
	That means,
	\begin{equation}\label{6}
	\langle \eta_{i}^{ \ell},\nu \rangle\leq 0, \ \ \forall \eta_{i}^{ \ell} \in  \ \  \partial^{T} \ell_{i}(k^{*}), \ \ \forall i \in I_{ \ell}
	\end{equation}
	\begin{equation}\label{7}
	\langle \eta_{i}^{ \hbar},\nu \rangle\leq 0, \ \ \forall \eta_{i}^{ \hbar} \in  \ \  \partial^{T} \hbar_{i}(k^{*}) \cup  \ \ \partial^{T}(-h_{i})(k^{*}) \ \ i=1,...,q,
	\end{equation}
	\begin{equation}\label{8}
	\langle \eta_{i}^{G},\nu \rangle\leq 0, \ \ \forall \eta_{i}^{G} \in  \ \  \partial^{T}G_{i}(k^{*}) \cup  \ \ \partial^{T}(-G_{i})(k^{*}) \ \  \forall i \in \Theta,
	\end{equation}
	\begin{equation}\label{9}
	\langle \eta_{i}^{H},\nu \rangle\leq 0, \ \ \forall \eta_{i}^{H} \in  \ \  \partial^{T}H_{i}(k^{*}) \cup  \ \ \partial^{T}(-H_{i})(k^{*}) \ \  \forall i \in \Upsilon,
	\end{equation}
	\begin{equation}\label{10}
	\langle \eta_{i}^{GH},\nu \rangle\leq 0, \ \ \forall \eta_{i}^{GH} \in  \ \  \partial^{T}(-G_{i})(k^{*}) \cup  \ \ \partial^{T}(-H_{i})(k^{*}) \ \  \forall i \in \Omega.
	\end{equation}
	By $(\ref{6}-\ref{10}),$ we get
	$$\nu \in  \ell^{-} \cap  \hbar^{-} \cap G^{-}_{\Theta} \cap H^{-}_{\Upsilon} \cap(GH)^{-}_{\Omega}=\Pi(k^{*}).$$
	Here, we use GS-ACQ at $k^{*}$ to obtain $\nu \in T(K,k^{*})$, for this we are choosing sequences $t_{z} \rightarrow 0$ and $\nu_{z} \rightarrow \nu$ such that $k^{*}+t_{z}\nu_{nz} \in K$ for each $z \in N.$ Since we initialy considerd that  $\mathcal{J}$ is localy Lipschitz at $k^{*}$ with some modulus $L > 0$ then by this property we have
	\begin{equation}\label{11}
	\mathcal{J}(k^{*}+t_{z}\nu_{z}) \leq  \mathcal{J}(k^{*}+t_{z}\nu)+L\Vert \nu_{z}-\nu\Vert , L > 0.
	\end{equation}
	Then, by relation (\ref{5}) and (\ref{11}), we get for sufficiently large $z$, such that 
	$$ \mathcal{J}(k^{*}+t_{z}\nu_{z}) <  \mathcal{J}(k^{*}).$$
	Above expression contradicts that $k^{*}$ is local minimizer of  $\mathcal{J}$. That means our initial assumption is correct i.e. $ 0 \in  \ \ \partial^{T} \mathcal{J}(k^{*})+\Delta $ is correct.
	Also we can say that, there exists vectors, $\lambda^{ \ell}_{I_{ \ell}} \geq 0, \ \  \lambda^{ \hbar}_{j}, \ \  j=1,...,q, \ \  \lambda^{G}_{i} \geq 0, \ \ i \in \Theta \cup \Omega, \ \  \lambda^{H}_{i} \geq 0, \ \ i \in \Upsilon \cup \Omega, \ \  \mu^{ \hbar}_{j}, \ \  j=1,...,q, \ \  \mu^{G}_{i}, \ \  i \in \Theta, \ \ \mu^{H}_{i}, i \in \Upsilon.$
	such that
	\begin{align}\label{12}
	0 \in \partial^{T} \mathcal{J}(k^{*})+&\sum_{i\in I_{ \ell}} \lambda_{i}^{ \ell}\partial^{T}( \ell_{i}(k^{*}))+\sum_{j= 1}^{q}[\lambda_{j}^{ \hbar}\partial^{T}( \hbar_{j}(k^{*}))+\mu_{j}^{ \hbar}\partial^{T}(- \hbar_{j})(k^{*})] \nonumber \\ &+\sum_{i \in \Theta \cup \Omega}\lambda_{i}^{G}\partial^{T}(-G_{i})(k^{*})+\sum_{i \in \Upsilon\cup\Omega}\lambda_{i}^{H}\partial^{T}(-H_{i})(k^{*})\nonumber\\&+\sum_{i \in \Theta}\mu_{i}^{G}\partial^{T}(G_{i})(k^{*})+\sum_{i \in \Upsilon}\mu_{i}^{H}\partial^{T}(H_{i})(k^{*}).
	\end{align}
	By taking $\lambda^{G}_{\Upsilon}=\lambda^{H}_{\Theta}= \mu^{G}_{\Upsilon \cup \Omega} =\mu^{H}_{\Theta \cup \Omega}=0$ on expression (\ref{12}), we have
	\begin{align*}
	0 \in \partial^{T} \mathcal{J}(k^{*})+&\sum_{i\in I_{ \ell}} \lambda_{i}^{ \ell}\partial^{T}( \ell_{i}(k^{*}))+\sum_{j= 1}^{q}[\lambda_{j}^{ \hbar}\partial^{T}( \hbar_{j}(k^{*}))+\mu_{j}^{ \hbar}\partial^{T}(- \hbar_{j})(k^{*})] \nonumber \\ &+\sum_{i=1}^{m}[\lambda_{i}^{G}\partial^{T}(-G_{i})(k^{*})+\lambda_{i}^{H}\partial^{T}(-H_{i})(k^{*})]\nonumber\\&+\sum_{i=1}^{m}[\mu_{i}^{G}\partial^{T}(G_{i})(k^{*})+\mu_{i}^{H}\partial^{T}(H_{i})(k^{*})],
	\end{align*}
	\begin{align*}
	&\lambda^{ \ell}_{I_{ \ell}} \geq 0, \ \  \lambda^{ \hbar}_{j}, \mu^{ \hbar}_{j} \geq 0, \ \  j=1,...,q, \ \ \lambda^{G}_{i},\lambda^{H}_{i},\mu^{G}_{i},\mu^{H}_{i}\geq 0, \ \ i=1,...,m,\\&\lambda^{G}_{\Upsilon}=\lambda^{H}_{\Theta}= \mu^{G}_{\Upsilon} =\mu^{H}_{\Theta}=0, \\& \mu^{H}_{i}=\mu^{G}_{i}=0 \ \  \forall i \in \Omega.
	\end{align*}
	Thus $k^{*}$ is GS-stationary point. Complete proof.
\end{proof}
By consequence of above theorem following corollary are also true:
\begin{corollary}
	Let $k^{*}$	be a local optimal solution of MPEC. Suppose  $\mathcal{J}$ is tangentially convex and locally Lipschitz at that point. Suppose effective constraint functions are tangentially convex at $k^{*}$. If $k^{*}$ holds GS-ACQ then $k^{*}$ is also GS-statioanry point.
\end{corollary}
\begin{proof}
	Since  $\mathcal{J}$ and effective constraint functions are tangentially convex at $k^{*}$ so it admits bounded tangential subdifferentials at $k^{*}$, then
	\begin{equation*}
	\Delta =cone \ \  \ \  \ell + cone \ \  \ \  \hbar + cone\ \   \ \  G_{\Theta} + cone  \ \   \ \ H_{\Upsilon} + cone \ \ \ \  (GH)_{\Omega},
	\end{equation*}
	is closed.	
\end{proof}
Our next example illustrates Theorem \ref{5.1}.
\begin{example}
	Let $\mathcal{J}$ be nondifferentiable objective function as,
	
	$(P) \ \ min \ \  \mathcal{J}(k_{1},k_{2})= |k_{1}|+k_{2}^{3}$
	
	such that $ \ell(k_{1},k_{2})=|k_{2}| \leq 0,  \ \  \hbar(k_{1},k_{2})=0,$
	
	$G(k_{1},k_{2})= k_{1} \geq 0 \ \ H(k_{1},k_{2})=k_{2} \geq 0.$\\
	
	Clearly, we can see that all the functions are tangentially convex function and $(0,0)$ is minimizer of MPEC (P). $\Theta, \ \ \Upsilon$ are empty but $\Omega$ is nonempty.\\
	The directonal derivatives of functions at $0=(0,0)$ in directon $d=(d_{1},d_{2})$ are,\\
	$ \mathcal{J}^{+}(0,d)=|d_{1}|,$\\
	$ \ell^{+}(0,d)=|d_{2}|,$\\
	$(-G)^{+}(0,d)=-d_{1},$\\
	$(-H)^{+}(0,d)=-d_{2}.$\\
	Tangential subdifferentials, are\\
	$\partial^{T} \mathcal{J}(0)=\{(-1,0),(1,0)\},$\\
	$\partial^{T} \ell(0)=\{(0,1),(0,-1)\},$\\
	$\partial^{T}(-G)(0)=\{(-1,0)\},$\\
	$\partial^{T}(-H)(0)=\{(0,-1)\}.$\\
	Hence, we have\\
	$\ell^{-}=\{ (d_{1},d_{2}) | \ \ d_{2}=0\},\ \  (GH)^{-}=\{ (d_{1},d_{2}) | \ \ d_{1} \geq 0, d_{2} \geq 0\}.$\\
	So, GS-ACQ is satisfied at $k^{*}=(0,0),$ i.e. $\Pi(0) \subset T(K,0).$ Also cone  $\Delta$ is closed. For $\lambda^{ \ell}=1 \geq 0, \ \ \lambda^{G}=1 \geq 0$ and $\lambda^{H}=1 \geq 0$ with $\mu^{G}=\mu^{H}=0,$ we have
	
	\begin{align*}
	0 \in \partial^{T} \mathcal{J}(0)&+\lambda^{ \ell}\partial^{T} \ell(0) +\lambda^{G}\partial^{T}(-G)(0)+\lambda^{H}\partial^{T}(-H)(0) \\&+\mu^{G}\partial^{T}G(0)+\mu^{H}\partial^{T}H(0).
	\end{align*}
	That means $0=(0,0)$ is GS-stationary point.
\end{example}
\begin{theorem}
	Suppose $k^{*}$ be a local minimum point of MPEC, and  $\mathcal{J}$ is locally Lipschitz function near $k^{*}.$ Suppose that $\mathcal{J}$ and all effective constraint functions are tangentially convex at $k^{*}$. Then, if $k^{*}$ satisfies MPEC-ACQ, then $k^{*}$ is definitely GA-stationary point.
\end{theorem}
\begin{proof}
	We claim that
	\begin{equation}\label{13}
	0 \in co \partial^{T} \mathcal{J}(k^{*})+\Lambda,
	\end{equation}
	where $\Lambda =cone \ \ ( \ell \cup  \hbar \cup G_{\Theta} \cup H_{\Upsilon} \cup (GH)_{\Omega}\ \cup G_{\Omega}).$\\
	Suppose by contradiction that (\ref{13}) is not holds. Since $\partial^{T} \mathcal{J}(k^{*})$  is compact and convex and $\Lambda$ is closed and convex. Then as similar proof of Theorem \ref{5.1}, we can find a non zero vector $\nu \in \mathbb{R}^{n}$  with $t_{z} \downarrow 0 \ \  and \ \  \nu_{z} \downarrow \nu$ s.t. for all large enough $z$,
	\begin{equation*}
	\mathcal{J}(k^{*}+t_{z}\nu_{z})<  \mathcal{J}(k^{*}),
	\end{equation*}
	which is contradiction of our initial assumption that $k^{*}$ is local minimum point of  $\mathcal{J}$. Now we observe that,
	\begin{equation*}
	\Lambda= \Delta+ cone \ \  (G_{\Omega}),
	\end{equation*}
	where $\Delta$ is same as we defined in previous Theorem \ref{5.1}.
	This gives us non-negative multipliers $\lambda^{\ell}_{i}, i \in I_{\ell}, \ \  \lambda^{ \hbar}_{j}, \ \  j=1,...,q, \ \  \lambda^{G}_{i}, \ \ i \in \Theta \cup \Omega, \ \  \lambda^{H}_{i}, \ \ i \in \Upsilon \cup \Omega, \ \  \mu^{ \hbar}_{j}, \ \  j=1,...,q, \ \  \mu^{H}_{i}, \ \  i \in \Upsilon, \ \ \mu^{G}_{i}, i \in \Theta \cup \Omega,$ such that
	\begin{align}\label{14}
	0 \in \partial^{T} \mathcal{J}(k^{*})+&\sum_{i\in I_{ \ell}} \lambda_{i}^{ \ell}\partial^{T}( \ell_{i}(k^{*}))+\sum_{j= 1}^{q}[\lambda_{j}^{ \hbar}\partial^{T}( \hbar_{j}(k^{*}))+\mu_{j}^{ \hbar}\partial^{T}(- \hbar_{j})(k^{*})] \nonumber \\ &+\sum_{i \in \Theta \cup \Omega}[\lambda_{i}^{G}\partial^{T}(-G_{i})(k^{*})+\lambda_{i}^{H}\partial^{T}(-H_{i})(k^{*})]\nonumber\\&+\sum_{i \in \Theta \cup \Omega}[\mu_{i}^{G}\partial^{T}(G_{i})(k^{*})+\mu_{i}^{H}\partial^{T}(H_{i})(k^{*})].
	\end{align}
	For $\lambda^{H}_{\Theta}=\lambda^{G}_{\Upsilon}= \mu^{G}_{\Upsilon}=\mu^{H}_{\Theta \cup \Omega}=0$ in (\ref{14}), we have
	\begin{align*}
	0 \in \partial^{T} \mathcal{J}(k^{*})+&\sum_{i\in I_{ \ell}} \lambda_{i}^{ \ell}\partial^{T}( \ell_{i}(k^{*}))+\sum_{j= 1}^{q}[\lambda_{j}^{ \hbar}\partial^{T}( \hbar_{j}(k^{*}))+\mu_{j}^{\hbar}\partial^{T}(- \hbar_{j})(k^{*})] \nonumber \\ &+\sum_{i=1}^{m}[\lambda_{i}^{G}\partial^{T}(-G_{i})(k^{*})+\lambda_{i}^{H}\partial^{T}(-H_{i})(k^{*})]\nonumber\\&+\sum_{i=1}^{m}[\mu_{i}^{G}\partial^{T}(G_{i})(k^{*})+\mu_{i}^{H}\partial^{T}(H_{i})(k^{*})],
	\end{align*}
	\begin{equation*}
	\lambda^{ \ell}_{I_{ \ell}}, \geq 0, \ \ \lambda^{ \hbar}_{j}, \ \  \mu^{ \hbar}_{j},  \ \  j=1,...,q, \ \  \lambda^{G}_{i}, \ \  \lambda^{H}_{i}, \ \  \mu^{G}_{i}, \ \ \mu^{H}_{i} \geq0, \ \ i=1,...,m,
	\end{equation*}
	$\lambda^{G}_{\Theta}=\lambda^{H}_{\Upsilon}= \mu^{G}_{\Upsilon}=\mu^{H}_{\Theta}=0,$\\
	
	$\mu^{H}_{i}=0  \ \ \forall i \in \Omega.$\\
	This shows that $k^{*}$ is GA stationary point of MPEC.
	This is the complete proof of our theorem.
\end{proof}

Our next theorem proof's that, GA-stationary condition will be (global or local) sufficient optimality condition under some special MPEC generalized convexity assumptions.
\begin{theorem}
	Let $k^{*}$ be the feasible point of MPEC and also satisfies the GA-stationary condition. Let we define the following index set at $k^{*}$,
	$$\Omega_{\mu}^{G}=\{i \in \Omega | \ \  \mu_{i}^{H} =0, \mu_{i}^{G}>0\},$$ $$\Omega_{\mu}^{H}=\{i \in \Omega | \ \  \mu_{i}^{G} =0, \mu_{i}^{H}>0\},$$ $$\Theta_{\mu}^{+}=\{i \in \Theta| \ \ \mu_{i}^{G}>0\},$$
	$$ \Upsilon_{\mu}^{+}=\{i \in \Upsilon| \ \ \mu_{i}^{H}>0\}.$$
	Let  $\mathcal{J}$ is $\partial^{T}$-pseudoconvax and $ \ell_{i}, (i \in I_{ \ell}), - \hbar_{i}, \ \  \hbar_{i}, (i=1,...,q), -G_{i}, (i \in \Theta \cup \Omega)$ and $-H_{i} (i \in \Upsilon \cup \Omega)$ are $\partial^{T}$-quasiconvex at $k^{*}$. Then, if  \ \ $\Omega_{\mu}^{G} \cup \mu_{i}^{H} \cup \Theta_{\mu}^{+} \cup \Upsilon_{\mu}^{+}= \phi$,
	so $k^{*}$ is global optimal solution of MPEC.
\end{theorem}
\begin{proof}
	Assume that $k$ be any feasible point of MPEC. Then $ \ell_{i}(k) \leq 0 =  \ell_{i}(k^{*})$ for all $i \in I_{ \ell},$ and by $\partial^{T}$-qausiconvaxity of $ \ell_{i}$ at $k^{*}$ we have,
	\begin{equation}\label{15}
	\langle \tau_{i},k-k^{*}\rangle \leq 0, \ \ \forall \tau_{i} \in \partial^{T} \ell_{i}(k^{*}) \ \ \forall i \in I_{ \ell}.
	\end{equation}
	By same argument, we have
	\begin{equation}
	\langle \zeta_{i},k-k^{*}\rangle \leq 0, \ \ \forall \zeta_{i} \in \partial^{T}\hbar_{i}(k^{*}) \ \ \forall i=1,...,q.
	\end{equation}
	\begin{equation}
	\langle \nu_{i},k-k^{*}\rangle \leq 0, \ \ \forall \nu_{i} \in \partial^{T}(-\hbar_{i})(k^{*}) \ \ \forall i=1,...,q.
	\end{equation}
	\begin{equation}
	\langle \sigma_{i},k-k^{*}\rangle \leq 0, \ \ \forall \sigma_{i} \in \partial^{T}(-G_{i})(k^{*}) \ \ \forall i \in \Theta \cup \Omega.
	\end{equation}
	\begin{equation}\label{19}
	\langle \xi_{i},k-k^{*}\rangle \leq 0, \ \ \forall \xi_{i} \in \partial^{T}(-H_{i})(k^{*}) \ \ \forall i \in \Upsilon \cup \Omega.
	\end{equation}
	Now if we consider $\Omega_{\mu}^{G} \cup \Omega_{\mu}^{H} \cup \Theta_{\mu}^{+} \cup \Upsilon_{\mu}^{+}= \phi$ and multiplying (\ref{15}-\ref{19}) by $\lambda_{i}^{ \ell} \ \ i \in I_{ \ell}, \ \  \lambda_{i}^{ \hbar}, \ \ \mu_{i}^{ \hbar}, i=1,...,q, \ \ \lambda_{i}^{G}, \ \  i \in \Theta \cup \Omega, \ \  \lambda_{i}^{H}, i \in \Upsilon \cup \Omega$, respectively, and adding all we have
	\begin{equation*}
	\Bigg\langle \sum_{i \in I_{ \ell}} \lambda_{i}^{ \ell}\tau_{i}+\sum_{j=1}^{q}[\lambda_{i}^{ \hbar}\zeta_{i}+\mu_{i}^{\hbar}\nu_{i}]+\sum_{j=1}^{m}[\lambda_{i}^{G} \sigma_{i}+\lambda_{i}^{H} \xi_{i}], k-k^{*} \Bigg\rangle\leq 0,
	\end{equation*}
	such that $\tau_{i} \in \partial^{T} \ell_{i}(k^{*}), \ \  \zeta_{i} \in \partial^{T} \hbar_{i}(k^{*}), \ \ \nu_{i} \in \partial^{T}(- \hbar_{i})(k^{*}), \ \ \sigma_{i} \in \partial^{T}(-G)_{i}(k^{*}), \ \  \xi_{i} \in \partial^{T}(-H)_{i}(k^{*})$. Initialy we assumed that $k^{*}$ satisfies GA-stationary condition, so we can select $\theta \in \partial^{T} \mathcal{J}(k^{*})$ such that $\langle \theta, k-k^{*} \rangle \geq 0$, then by $\partial^{T}$-pseudoconvexity of  $\mathcal{J}$ at $k^{*}$ we have $ \mathcal{J}(k) \geq  \mathcal{J}(k^{*})$ for all feasible point $k$. Hence $k^{*}$ is global minimum point of MPEC (P) problem.
\end{proof}

\section{Conclusion}\label{sec13}
In this paper, we obtained optimality conditions and various relationships between MPEC constraint qualifications in nonsmooth MPEC, using tangential subdifferentials for the class of tangentially convex functions. We established some useful results that confirm the optimality of any feasible point for nonsmooth mathematical programming problems with equilibrium constraints (MPEC). Using the concept of tangential subdifferentials, we can also modify past results where other subdifferentials are unable. Mathematical programs with equilibrium constraints are very important in fields like networking problems, applied mathematics, engineering, economics, and nonsmooth modeling, etc.

\section*{Declarations}
\subsection*{Ethical Approval}
Not Applicable.

\subsection*{Competing Interest}
The authors declare that they have no competing interests.
\subsection*{Author's Contributions}
The authors declare that the study was realized in collaboration with equal responsibility. All authors read and approved the final manuscript.
\subsection*{Funding}
The first author is financially supported by Research Grant for Faculty (IoE) under Dev. Scheme No. 6031. The second author is financially supported by CSIR-UGC JRF, New Delhi, India through Reference No. 20161014569.
\subsection*{Availability of Data and Materials}
Data sharing not applicable to this paper as no data sets were generated or analyzed during the current study.
\subsection*{Consent for participate}
Not Applicable.
\subsection*{Consent for publication}
Not Applicable.
\subsection*{Code availability}
Not Applicable.



\end{document}